\documentclass[11pt,leqno]{article}
\usepackage[margin=1in]{geometry} 
\usepackage{amssymb,amsfonts,amsmath,bbm,mathrsfs,stmaryrd,mathtools}
\usepackage{xcolor}
\usepackage{url}

\usepackage{graphicx}

\usepackage{accents}

\usepackage{extarrows}

\usepackage[shortlabels]{enumitem}
\usepackage{tensor}

\usepackage{xr}
\usepackage[T1]{fontenc}
\usepackage[utf8]{inputenc}
\usepackage[colorlinks,
linkcolor=black!75!red,
citecolor=blue,
pdftitle={},
pdfproducer={pdfLaTeX},
pdfpagemode=None,
bookmarksopen=true,
bookmarksnumbered=true,
backref=page]{hyperref}

\usepackage{tikz}
\usetikzlibrary{arrows,calc,decorations.pathreplacing,decorations.markings,decorations.shapes,intersections,shapes.geometric,through,fit,shapes.symbols,positioning,decorations.pathmorphing}

\makeatletter
\newlength\zig@L
\newlength\zig@La
\newlength\zig@Lb

\newcommand{\xzigrightarrow}[2][]{%
  \mathrel{%
    \settowidth{\zig@La}{$\scriptstyle #2$}%
    \settowidth{\zig@Lb}{$\scriptstyle #1$}%
    \zig@L=\zig@La\relax
    \ifdim\zig@Lb>\zig@L \zig@L=\zig@Lb\fi
    \advance\zig@L by 2.2em\relax
    \tikz[baseline=-0.65ex]{%
      \draw[->,
            line cap=round,
            decorate,
            decoration={zigzag,segment length=4pt,amplitude=1.1pt}]%
        (0,0) -- (\zig@L,0)
        node[midway,above=2pt] {$\scriptstyle #2$}%
        \if\relax\detokenize{#1}\relax\else
          node[midway,below=2pt] {$\scriptstyle #1$}%
        \fi
      ;
    }%
  }%
}
\makeatother

\makeatletter
\newcommand{\squigjoin}{1mu} % tune this: 0.5mu, 1mu, 1.5mu, ...

\def\sqleft@{\sim}                    % no overlap here
\def\sqmid@{\sim\mkern-\squigjoin}    % overlap only between repeated mids

\def\rightsquigarrowfill@{%
  \arrowfill@{\sqleft@}{\sqmid@}{\mkern-4mu\succ}%
}

\newcommand{\xrightsquigarrow}[2][]{%
  \ext@arrow 0359\rightsquigarrowfill@{#1}{#2}%
}
\makeatother

\makeatletter
\newcommand*\circled[1]{\tikz[baseline=(char.base)]{
    \node[shape=circle, draw, inner sep=0pt, 
    minimum height={\f@size},] (char) {\vphantom{WAH1g}#1};}}
\makeatother

\makeatletter
\DeclareRobustCommand\widecheck[1]{{\mathpalette\@widecheck{#1}}}
\def\@widecheck#1#2{%
    \setbox\z@\hbox{\m@th$#1#2$}%
    \setbox\tw@\hbox{\m@th$#1%
       \widehat{%
          \vrule\@width\z@\@height\ht\z@
          \vrule\@height\z@\@width\wd\z@}$}%
    \dp\tw@-\ht\z@
    \@tempdima\ht\z@ \advance\@tempdima2\ht\tw@ \divide\@tempdima\thr@@
    \setbox\tw@\hbox{%
       \raise\@tempdima\hbox{\scalebox{1}[-1]{\lower\@tempdima\box
\tw@}}}%
    {\ooalign{\box\tw@ \cr \box\z@}}}
\makeatother

\usepackage{braket}

\usepackage[amsmath,thmmarks,hyperref]{ntheorem}
\usepackage{cleveref}

\newcommand\nthalias[1]{\AddToHook{env/#1/begin}{\crefalias{lemma}{#1}}}

\nthalias{assumption}
\nthalias{conjecture}
\nthalias{construction}
\nthalias{convention}
\nthalias{corollaryN}
\nthalias{corollary}
\nthalias{definition}
\nthalias{examples}
\nthalias{example}
\nthalias{lemma}
\nthalias{notation}
\nthalias{propositionN}
\nthalias{proposition}
\nthalias{question}
\nthalias{recollection}
\nthalias{remarks}
\nthalias{remark}
\nthalias{sketch}
\nthalias{theoremN}
\nthalias{theorem}

\creflabelformat{enumi}{#2#1#3}

\crefname{section}{Section}{Sections}
\crefformat{section}{#2Section~#1#3} 
\Crefformat{section}{#2Section~#1#3} 

\crefname{subsection}{\S}{\S\S}
\AtBeginDocument{%
  \crefformat{subsection}{#2\S#1#3}%
  \Crefformat{subsection}{#2\S#1#3}% 
}

\crefname{subsubsection}{\S}{\S\S}
\AtBeginDocument{%
  \crefformat{subsubsection}{#2\S#1#3}%
  \Crefformat{subsubsection}{#2\S#1#3}% 
}

\theoremstyle{plain}

\newtheorem{lemma}{Lemma}[section]
\newtheorem{proposition}[lemma]{Proposition}
\newtheorem{corollary}[lemma]{Corollary}
\newtheorem{theorem}[lemma]{Theorem}

\theoremstyle{plain}
\theoremnumbering{Alph}

\theoremstyle{plain}
\theorembodyfont{\upshape}
\theoremsymbol{\ensuremath{\blacklozenge}}

\newtheorem{definition}[lemma]{Definition}

\newtheorem{remark}[lemma]{Remark}
\newtheorem{remarks}[lemma]{Remarks}

\newtheorem{notation}[lemma]{Notation}

\newtheorem{recollection}[lemma]{Recollection}

\crefname{definition}{definition}{definitions}
\crefformat{definition}{#2definition~#1#3} 
\Crefformat{definition}{#2Definition~#1#3} 

\crefname{ex}{example}{examples}
\crefformat{example}{#2example~#1#3} 
\Crefformat{example}{#2Example~#1#3} 

\crefname{exs}{example}{examples}
\crefformat{examples}{#2example~#1#3} 
\Crefformat{examples}{#2Example~#1#3} 

\crefname{remark}{remark}{remarks}
\crefformat{remark}{#2remark~#1#3} 
\Crefformat{remark}{#2Remark~#1#3} 

\crefname{remarks}{remark}{remarks}
\crefformat{remarks}{#2remark~#1#3} 
\Crefformat{remarks}{#2Remark~#1#3} 

\crefname{convention}{convention}{conventions}
\crefformat{convention}{#2convention~#1#3} 
\Crefformat{convention}{#2Convention~#1#3} 

\crefname{notation}{notation}{notations}
\crefformat{notation}{#2notation~#1#3} 
\Crefformat{notation}{#2Notation~#1#3} 

\crefname{table}{table}{tables}
\crefformat{table}{#2table~#1#3} 
\Crefformat{table}{#2Table~#1#3}

\crefname{lemma}{lemma}{lemmas}
\crefformat{lemma}{#2lemma~#1#3} 
\Crefformat{lemma}{#2Lemma~#1#3} 

\crefname{proposition}{proposition}{propositions}
\crefformat{proposition}{#2proposition~#1#3} 
\Crefformat{proposition}{#2Proposition~#1#3} 

\crefname{propositionN}{proposition}{propositions}
\crefformat{propositionN}{#2proposition~#1#3} 
\Crefformat{propositionN}{#2Proposition~#1#3} 

\crefname{corollary}{corollary}{corollaries}
\crefformat{corollary}{#2corollary~#1#3} 
\Crefformat{corollary}{#2Corollary~#1#3} 

\crefname{corollaryN}{corollary}{corollaries}
\crefformat{corollaryN}{#2corollary~#1#3} 
\Crefformat{corollaryN}{#2Corollary~#1#3} 

\crefname{theorem}{theorem}{theorems}
\crefformat{theorem}{#2theorem~#1#3} 
\Crefformat{theorem}{#2Theorem~#1#3} 

\crefname{theoremN}{theorem}{theorems}
\crefformat{theoremN}{#2theorem~#1#3} 
\Crefformat{theoremN}{#2Theorem~#1#3} 

\crefname{enumi}{}{}
\crefformat{enumi}{#2#1#3}
\Crefformat{enumi}{#2#1#3}

\crefname{assumption}{assumption}{Assumptions}
\crefformat{assumption}{#2assumption~#1#3} 
\Crefformat{assumption}{#2Assumption~#1#3} 

\crefname{construction}{construction}{Constructions}
\crefformat{construction}{#2construction~#1#3} 
\Crefformat{construction}{#2Construction~#1#3} 

\crefname{sketch}{sketch}{Sketches}
\crefformat{sketch}{#2sketch~#1#3} 
\Crefformat{sketch}{#2Sketch~#1#3} 

\crefname{recollection}{recollection}{Recollections}
\crefformat{recollection}{#2recollection~#1#3} 
\Crefformat{recollection}{#2Recollection~#1#3} 

\crefname{question}{question}{Questions}
\crefformat{question}{#2question~#1#3} 
\Crefformat{question}{#2Question~#1#3} 

\crefname{equation}{}{}
\crefformat{equation}{(#2#1#3)} 
\Crefformat{equation}{(#2#1#3)}

\numberwithin{equation}{section}

\theoremstyle{nonumberplain}
\theoremsymbol{\ensuremath{\blacksquare}}

\newtheorem{proof}{Proof}
\newcommand\pf[1]{\newtheorem{#1}{Proof of \Cref{#1}}}

\newcommand\bC{{\mathbb C}}
\newcommand\bD{{\mathbb D}}

\newcommand\bQ{{\mathbb Q}}
\newcommand\bR{{\mathbb R}}
\newcommand\bS{{\mathbb S}}
\newcommand\bT{{\mathbb T}}

\newcommand\bZ{{\mathbb Z}}

\newcommand\cA{{\mathcal A}}

\newcommand\cE{{\mathcal E}}
\newcommand\cF{{\mathcal F}}

\newcommand\cL{{\mathcal L}}

\newcommand\cS{{\mathcal S}}

\newcommand\cU{{\mathcal U}}

\newcommand\cY{{\mathcal Y}}

\newcommand\wt{\widetilde}

\DeclareMathOperator{\Ad}{Ad}
\DeclareMathOperator{\id}{id}

\DeclareMathOperator{\End}{\mathrm{End}}

\DeclareMathOperator{\Max}{\mathrm{Max}}

\DeclareMathOperator{\diag}{\mathrm{diag}}
\DeclareMathOperator{\Emb}{\mathrm{Emb}}

\newcommand{\qedhere}{\mbox{}\hfill\ensuremath{\blacksquare}}

\newcommand{\comment}[1]{}

\newcommand{\xrightarrowdbl}[2][]{%
  \xrightarrow[#1]{#2}\mathrel{\mkern-14mu}\rightarrow
}

\title{Irremediably singular quantum branched covers}
\author{Alexandru Chirvasitu}

\begin{document}

\date{}

\newcommand{\Addresses}{{% additional braces for segregating \footnotesize
  \bigskip
  \footnotesize

  \textsc{Department of Mathematics, University at Buffalo}
  \par\nopagebreak
  \textsc{Buffalo, NY 14260-2900, USA}  
  \par\nopagebreak
  \textit{E-mail address}: \texttt{achirvas@buffalo.edu}

}}

\maketitle

\begin{abstract}
  We prove that the Natsume-Olsen non-commutative spheres $\mathbb{S}^{2n-1}_{\theta}$ dualize for rational deformation parameters to provide examples of quantum branched covers over their respective centers' maximal spectra, embeddable into locally trivial $C^*$ bundles precisely when said spheres are classical; this, despite the bundles in question being of finite type over dimensional strata. The proof of the noted rigidity phenomenon relies on the quantum analogue of the 3-sphere's Heegaard splitting as a pushout of two solid tori.
\end{abstract}

\noindent \emph{Key words:
  $C^*$ bundle;
  Heegaard decomposition;
  homotopy;
  locally trivial;
  matrix bundle;
  quantum sphere;
  quantum torus;
  subhomogeneous
}

\vspace{.5cm}

\noindent{MSC 2020: 46M20 55R10; 46L85; 46L65; 53B21; 46L30; 55R40; 55S40
  
  % 46M20  	Methods of algebraic topology in functional analysis (cohomology, sheaf and bundle theory, etc.)
  % 55R10  	Fiber bundles in algebraic topology
  % 46L85  	Noncommutative topology
  % 46L65  	Quantizations, deformations for selfadjoint operator algebras
  % 53B21  	Methods of local Riemannian geometry
  % 46L30  	States of selfadjoint operator algebras
  % 55R40  	Homology of classifying spaces and characteristic classes in algebraic topology
  % 55S40  	Sectioning fiber spaces and bundles in algebraic topology
  
}

%\tableofcontents

%%%%%%%%%%%%%%%%%%%%%%%%%%%%%%%%
%%%%%%%%%%%%%%%%%%%%%%%%%%%%%%%%
\section*{Introduction}

In studying the internal structure a $C^*$-algebra acquires by virtue of having large central subalgebras one is naturally led to the notion of a \emph{field of $C^*$-algebras} or \emph{$C^*$ bundles} for short (\cite[Introduction]{fell_struct}, \cite[Definition II.13.4]{fd_bdl-1} or \cite[pp.7-9]{dg_ban-bdl} for the broader notion of a Banach bundle, etc.) : continuous open surjections $\cA\xrightarrowdbl{\pi}X$ with $C^*$ structures on the individual \emph{fibers} $\cA_x:=\pi^{-1}x$ so that the $*$-algebraic operations and the fiber-wise norms vary continuously in ways not difficult to make precise. 

The present paper is concerned exclusively with \emph{subhomogeneous} Banach/$C^*$ bundles \cite[Definition IV.1.4.1]{blk}: those with a global finite bound on the fiber dimensions $\dim \cA_x$. It is in that context that \cite{bg_cx-exp} studies the \emph{non-commutative branched covers} introduced in \cite[Definition 1.2]{pt_brnch}: subhomogeneous $C^*$ bundles (typically over compact Hausdorff base spaces) with the additional requirement that there be
\begin{itemize}[wide]
\item a \emph{conditional expectation} (norm-1 projection \cite[Theorem II.6.10.2]{blk})
  \begin{equation*}
    \text{continuous sections of $\cA$}
    =:
    \Gamma(\cA)\xrightarrowdbl{\quad E\quad}C(X)
    :=
    \text{continuous $\bC$-valued functions on $X$};
  \end{equation*}

\item of \emph{finite index} in the sense of \cite[Definition 2]{fk_fin-ind}: $KE-\id$ is a positive map for some $K\ge 1$.    
\end{itemize}
As homogeneity (i.e. fiber-dimension \emph{constancy} rather than boundedness) automatically provides such expectations \cite[Proposition 3.4]{bg_cx-exp}, whether or not a subhomogeneous bundle is a quantum branched cover can be construed as a gauge for how pathologically the singular (lower-than-typical-dimensional) fibers are distributed. Another such gauge is \emph{homogeneous embeddability}: whether or not a subhomogeneous $C^*$ bundle is embeddable into a homogeneous one (or into a locally trivial matrix bundle: the framing employed throughout most of the sequel). Reasons why such embeddings might not be possible are not difficult to glean: 
\begin{itemize}[wide]
\item On the one hand, \cite[Example 3.6]{bg_cx-exp} provides a quantum branched cover over a one-point compactification $X^+=X\sqcup\{*\}$ with one exceptional fiber $\bC$ at the one exceptional point $*$, typical fiber $M_2$ elsewhere, and failing matrix-bundle embeddability because the bundle is not \emph{of finite type} on $X$: not trivializable over the members of a finite open cover.

\item On the other hand, in otherwise very tame topological situations a bundle with one exceptional $\bC^2$ fiber and $M_3$ generic fiber fails matrix-bundle embeddability because it fails to admit a faithful tracial expectation (\cite[Example 2.2]{CHIRVASITU2026130746}, \cite[Proposition 2.1]{2605.10752v1}). 
\end{itemize}

It becomes natural at this stage to ask whether, such obstacles being absent, matrix-bundle embeddability can fail for a quantum branched cover $\cA\xrightarrowdbl{}X$ for reasons more directly traceable to topological-algebraic properties of the locally trivial bundles obtained by restricting $\cA$ to its \emph{strata}: loci of constant fiber isomorphism class. With that object in view we examine subhomogeneous $C^*$ bundles attached to various quantum-manifold function algebras. Recall the \emph{quantum} (or \emph{non-commutative}) \emph{tori} $\bT_{\theta}^n$ and \emph{spheres} $\bS_{\theta}^{2n-1}$ of, say, \cite[p.193]{rief_case} and \cite[Definition 2.1]{no_sph} respectively: defined dually by describing their associated (non-commutative) $C^*$ complex-function algebras by generators/relations: for $\theta=-\theta^t\in M_n(\bR)$,
\begin{equation*}
  \begin{aligned}
    A^n_{\theta}:=C(\bT^n_{\theta})
    &:=
      C^*\Braket{\text{unitary }U_{1\le i\le n}
      \big|
      \forall(i,j)\left(U_j U_i = e^{2\pi i \theta_{ij}}U_iU_j\right)}\\
    C^n_{\theta}:=C(\bS^{2n-1}_{\theta})
    &:=
      C^*\Braket{\text{normal }T_{1\le i\le n}
      \big|
      \forall(i,j)\left(T_j T_i=
      e^{2\pi i \theta_{ij}}T_iT_j\right)
      ,\quad
      \sum_i T_i^*T_i=1}.
  \end{aligned}
\end{equation*}

The second displayed family, it turns out, provides examples ``in nature'' (for rational deformation parameters $\theta$) of precisely the type alluded to above. The main result to that effect reads as follows.

\begin{theorem}\label{th:sn.theta}
  Let $n\in \bZ_{\ge 2}$ and $\theta\in M_n(\bQ)$ a rational skew-symmetric matrix.
  \begin{enumerate}[(1),wide]
  \item\label{item:th:sn.theta:nc.brnch} The center inclusion $Z(C^n_{\theta})\le C^n_{\theta}$ is dual to a non-commutative branched cover $\cS\xrightarrowdbl{\pi}X$.

  \item\label{item:th:sn.theta:ft} $\cS$ has finite type over its dimensional \emph{strata}
    \begin{equation*}
      X_d:=\left\{x\in X\ :\ \dim \cS_x=d\right\}. 
    \end{equation*}

  \item\label{item:th:sn.theta:tr.exp} $\cS$ is equipped with a unique tracial expectation
    \begin{equation*}
      \Gamma(\cS)=C^n_{\theta}
      \xrightarrowdbl{\quad E\quad}
      Z(C^n_{\theta})=C(X),
    \end{equation*}
    \emph{optimally} of finite index in the sense that $KE-\id\ge 0$ for the theoretically minimal value
    \begin{equation}\label{eq:K.ct}
      K=\sup_{x\in X}\sum\left(\text{dimensions of irreducible $\cS_x$-representations}\right).
    \end{equation}

  \item\label{item:th:sn.theta:not.emb} $\cS$ embeds into a locally trivial subhomogeneous $C^*$ bundle over $X$ if and only if $\theta$ is integral, i.e. the quantum sphere $\bS^{2n-1}_{\theta}$ is classical. 
  \end{enumerate}
\end{theorem}

The proof reduces fairly quickly via extant literature to the case $n=2$ (3-spheres), where some of the algebraic-topology machinery alluded to above becomes serviceable.

% %

%%%%%%%%%%%%%%%%%%%%%%%%%%%%%%%%
%%%%%%%%%%%%%%%%%%%%%%%%%%%%%%%%
\section{Quantum spheres and singular solid non-commutative tori}\label{se:qtm.mflds}

For rational $\theta=\left(\theta_{ij}=\frac{p_{ij}}{q_{ij}}\right)_{i,j}\in M_n(\bQ)$ we have (\cite[\S 2]{hks_erg} or \cite[Theorem 3.1]{rief_canc} for $n=2$, \cite[Proposition 2.6]{MR5049282}, say, in general)
\begin{equation}\label{eq:a.is.end.e}
  A^n_{\theta}
  \cong
  \End(\cE)
  =
  \Gamma\left(\cE nd \left(\cE\right)\right)
  =
  \Gamma\left(\cE\otimes \cE^*\right)
\end{equation}
for a \emph{projectively flat} \cite[Corollary I.2.7]{kob_cplx} vector bundle $\cE$ over the classical torus $\bT^n$ recoverable as the spectrum $\Max Z(A^n_{\theta})$ of the center $Z(A^n_{\theta})$. The rank $r=r_{\theta}$ of $\cE$ is given explicitly in \cite[Lemma 2.1]{MR5049282}, and is precisely the lowest-terms denominator of $\theta_{12}=\frac pq$ for $n=2$. The identification
\begin{equation} \label{eq:cntheta}
  C^n_{\theta}
  \ni T_i
  \xmapsto{\quad}
  t_i U_i
  \in
  C\left(\bS^{n-1}_{\ge 0},A^n_{\theta}\right)
  ,\quad
  \bS^{n-1}_{\ge 0}
  :=
  \left\{\left(t_1,\cdots,t_n\right)\in \bR^n_{\ge 0}\ :\ \sum t_i^2=1\right\}
\end{equation}
is an embedding \cite[Theorem 2.5]{no_sph} giving $C^n_{\theta}$ its own realization as $\Gamma(\cA)$ for a subhomogeneous $C^*$ bundle $\cA\xrightarrowdbl{}\Max Z\left(C^n_{\theta}\right)$ with generic fiber $M_{r_{\theta}}$. For $n=2$ and rational
\begin{equation*}
  \theta=
  \begin{pmatrix}
    0&\theta_{12}\\
    -\theta_{12}&0
  \end{pmatrix}
  ,\quad
  \theta_{12}=\frac{p}{q}
  ,\quad
  p\in \bZ,\ q\in \bZ_{>0},\ \gcd(p,q)=1
\end{equation*}
the situation is particularly pleasant \cite[Proposition 4.5]{MR5040755}:

\begin{recollection}\label{rec:s3t2}
  \begin{enumerate}[(1),wide]
  \item\label{item:rec:s3t2:a} $A^2_{\theta}\cong \Gamma(\cA:=\cE nd \cE)$ is a $q\times q$-matrix bundle over
    \begin{equation*}
      \bT^2
      \cong
      \Max Z(A^2_{\theta})
      =
      \Max C^*\Braket{U_1^q,U_2^q}.
    \end{equation*}

  \item\label{item:rec:s3t2:cz} \Cref{eq:cntheta} then realizes $C^2_{\theta}$ as $\Gamma(\cS)$ for a subhomogeneous $C^*$-bundle $\cS\xrightarrowdbl{}\bS^3\cong \Max Z^2_{\theta}:=Z\left(C^2_{\theta}\right)$ over the classical sphere. To expand: identifying $\bS^{1}_{\ge 0}\cong [0,1]$ (and conflating $C^2_{\theta}$ with an $A_{\theta}^2$-valued function algebra on $[0,1]$ via \Cref{eq:cntheta}), we have
    \begin{equation*}
      \begin{aligned}
        C^2_{\theta,t}
        &:=
          C^2_{\theta}/\left(\text{$t$-vanishing functions}\right)
          \cong
          \begin{cases}
            A^2_{\theta}&t\in (0,1)\\
            C^*\Braket{U_{t+1}}&t\in \{0,1\}
          \end{cases}\\
        Z^2_{\theta,t}
        &:=
          Z^2_{\theta}/\left(\text{$t$-vanishing functions}\right)
          \cong
          \begin{cases}
            Z\left(A^2_{\theta}\right)&t\in (0,1)\\
            C^*\Braket{U^q_{t+1}}&t\in \{0,1\}.
          \end{cases}
      \end{aligned}
    \end{equation*}
    Regarding $U^q_i$ as coordinates on the classical torus $\bT^2$, there are corresponding restrictions $\cS_t\xrightarrowdbl{}\bS^3_t$ with the latter space being the fiber at $t\in [0,1]$ of the map
    \begin{equation}\label{eq:s301}
      \begin{tikzpicture}[>=stealth,auto,baseline=(current  bounding  box.center)]
        \path[anchor=base] 
        (0,0) node (l) {$\bS^3$}
        +(4,.5) node (u) {$\bS^2$}
        +(6,0) node (r) {$[0,1]$}
        ;
        \draw[->>] (l) to[bend left=6] node[pos=.5,auto] {$\scriptstyle \text{Hopf fibration}$} node[pos=.5,auto,swap]{$\scriptstyle \text{\cite[\S 14.1.9]{td_alg-top}}$} (u);
        \draw[->>] (u) to[bend left=6] node[pos=.5,auto] {$\scriptstyle $} (r);
        \draw[->>] (l) to[bend right=10] node[pos=.5,auto,swap] {$\scriptstyle \eta$} (r);
      \end{tikzpicture}
    \end{equation}
    resulting from specializing \Cref{eq:cntheta} to centers; the right-hand arrow simply compresses the circles featuring in the bulk of a decomposition $\bS^2\cong \left\{p_0,p_1\right\}\sqcup (0,1)\times \bS^1$ to points for a choice of antipodes $p_{0,1}$. 
    
  \item\label{item:rec:s3t2:sing} The singular fibers of $\cS$ are abelian $q$-dimensional and lie above the two core circles (henceforth $\bS^1_{0,1}$) of the two solid tori in the standard \emph{Heegaard decomposition} \cite[Exercise 2.11(i)]{hemp_3mfld}
    \begin{equation}\label{eq:heeg.dec}
      \bS^3
      \cong
      \left(\bD^2\times \bS^1\right)
      \cup_{\bT^2}
      \left(\bS^1\times \bD^2\right).
    \end{equation}
    
  \item\label{item:rec:s3t2:res} The restriction $\cS_t:=\cS|_{\bT_t^2}$ to any 2-torus slice
    \begin{equation*}
      \bT^2_t:=
      \bT^2\times \{t\}
      \subset
      \bT^2\times (0,1)
      \cong
      \bS^3\setminus \left(\bS^1_0\sqcup \bS^1_1\right)
      ,\quad
      t\in (0,1)
    \end{equation*}
    is isomorphic to the selfsame $\cA\xrightarrowdbl{}\bT^2$. 
  \end{enumerate}
\end{recollection}

\begin{remarks}\label{res:bdl.on.s3}
  \begin{enumerate}[(1),wide]
  \item Given the centrality of the 3-sphere in the preceding discussion, we remind the reader that all bundles with Lie structure group over $\bS^3$ are trivial. This follows from the classification \cite[Theorem 18.5]{steen_fib_1999} of fiber bundles over spheres in conjunction with
    \begin{itemize}[wide]
    \item the simple connectivity of $\bS^3$, ensuring structure-group reduction to \emph{connected} Lie groups;
    \item reduction to maximal compact subgroups therefrom (as recalled in \cite[Remark 1.12(2)]{2605.10752v1} for instance);
    \item and the fact \cite[Proposition V.7.5]{btd_lie_1995} that compact Lie groups (hence also all Lie groups) have vanishing $\pi_2$. 
    \end{itemize}

  \item We remark also that matrix bundles $\cA$ on 2-tori are automatically of the form $\cE nd(\cE)\cong \cE\otimes \cE$ for vector bundles $\cE$, uniquely determined up to tensoring by line bundles.

    Indeed: the \emph{Dixmier-Douady class} \cite[post Definition 9.5]{gbvf_ncg} $\alpha(\cA)\in H^3(\bT^2,\bZ)$ vanishes for obvious dimension reasons, hence \cite[Remark 18.3.8]{hjjm_bdle} the existence of $\cE$. As to the uniqueness claim, it is a consequence of the \emph{fiber-sequence} fragment
    \begin{equation*}
      B\bS^1
      \xrightarrow{\quad}
      BU(\bullet)
      \xrightarrow{\quad}
      BPU(\bullet)
    \end{equation*}
    of classifying spaces attached to
    \begin{equation*}
      \{1\}
      \to
      \bS^1
      \lhook\joinrel\xrightarrow{\quad}
      U(\bullet)
      \xrightarrowdbl{\quad}
      PU(\bullet)
      \to
      \{1\}:
    \end{equation*}
    The right-hand map classifies $\cE\mapsto \cE\otimes \cE^*$ while $B\bS^1$ classifies line bundles. Alternatively:
    \begin{itemize}[wide]
    \item Two $q\times q$-matrix bundles over a 2-torus are isomorphic precisely when their corresponding characteristic classes (the $\beta_q\in H^2(\bT^2,\bZ/q)$ of \cite[\S 18.3.7]{hjjm_bdle}) are equal; this is noted in broader generality as part of \cite[Theorem 0.1(2)]{2509.10812v2}.

    \item This in turn means that $\cE\otimes \cE^*\cong \cF\otimes \cF^*$ if and only if the first Chern classes of the two rank-$q$ vector bundles are equal modulo $q$.

    \item Finally, the complete topological characterization \cite[p.2, Proposition]{MR1423157} of a vector bundle over a surface by Chern class and rank renders this equivalent to $\cF\cong \cE\otimes \cL$ for a line bundle $\cL$.
    \end{itemize}
  \end{enumerate}  
\end{remarks}

\begin{theorem}\label{th:s.not.ext}
  Let $\theta_{12}=\frac{p}{q}\in \bQ$ be a lowest-terms rational.
  
  The subhomogeneous $C^*$ bundle $\cS\xrightarrowdbl{} \bS^3$ of \Cref{rec:s3t2}\Cref{item:rec:s3t2:cz} associated to the quantum sphere $\bS^3_{\theta}$ embeds into a locally trivial subhomogeneous $C^*$ bundle over either of the two solid tori $\bD^2\times \bS^1$ in the decomposition \Cref{eq:heeg.dec} if and only if $q=1$. 
\end{theorem}

We retain the notation employed in \Cref{rec:s3t2}, and supplement it with subscripts $F\subseteq [0,1]$ for both $\bS^3$ and the bundle $\cS$ to indicate $\eta^{-1}F$ (see \Cref{eq:s301}) and the restriction $\cS|_{\eta^{-1}F}$ respectively. The elements $U_i^q\in A^2_{\theta}$ can be regarded as coordinates on any of the tori $\bT^2_t$, $t\in (0,1)$, with $U^q_{t+1}$ also acting as a respective coordinate along the exceptional circle $\bS^1_{t\in \{0,1\}}$.

Recall that per the bundles-on-spheres classification of \cite[Theorem 18.5]{steen_fib_1999} bundles over $\bS^1$ (hence also the homotopy-equivalent solid tori) with path-connected structure group are automatically trivial. This applies in particular to vector and matrix bundles, their structure groups being unitary and projective unitary respectively.

It will be profitable to isolate a concrete technical embeddability criterion. We first need a bit of notation referring to shifted diagonal matrices. 

\begin{notation}\label{not:off.diag}
  Write
  \begin{equation*}
    \diag_k\left(z_1\cdots z_n\right)
    :=
    M=\left(m_{ij}\right)_{i,j=1}^n
    ,\quad
    m_{ij}
    :=
    \delta_{j-i,k}z_i
    \quad
    \left(\text{$\delta_{\bullet}:=$Kronecker delta}\right),
  \end{equation*}
  with the comparison between $j-i$ and $k$ being modulo $n$ (so plain diagonal matrices correspond to index 0: $\diag=\diag_0$). 
\end{notation}

Throughout the ensuing discussion $\Ad_T$ denotes conjugation $T\bullet T^{-1}$. Note also, in preparation for stating \Cref{pr:embs.are.homotopies}, that embeddings $M_q\le M_n$ (and hence of $M_q$ bundles into $M_n$ bundles) exist only if $q|n$. For a reduced rational $\frac pq$ (featuring prominently in the sequel as the torus/sphere-deformation parameter) fix
\begin{equation}\label{eq:diag.subdiag}
  D:=\diag_0\left(\zeta^i\right)_{i=0}^{q-1}
  ,\ \zeta
  :=
  e\left(\frac pq\right)
  :=
  \exp\left(2\pi i\cdot \frac{p}{q}\right).
\end{equation}
For a finite-dimensional $C^*$-algebra $A$ we write $\Emb_{A,M_n}$ for the space of unital $C^*$-embeddings $A\le M_n$ and $\Emb^{=}_{A,M_n}$ for only those unital embeddings giving all simple factors of $A$ equal multiplicities; $\Emb^=_{M_q,M_n}$ is non-empty, then, precisely when $q|n$.

\begin{remark}\label{re:emb.on.emb.fib}
  Note the homogeneity of both $\Emb^=_{\bullet,M_n}$, $\bullet\in \{C^*(D),M_q\}$ under the conjugation action of the unitary group $U(n)=U(qd)$ when $q|n$ (so that both manifolds are non-empty). Writing $C_{G}(A)$ for the centralizer of $A$ in a group $G$ of automorphisms thereof (whatever structure $A$ may be), the restriction map
  \begin{equation}\label{eq:homog.sp.emb.fib}
    \Emb^=_{M_q, M_q\otimes M_d}
    \xrightarrowdbl{\quad\mathrm{res}\quad}
    \Emb^=_{C^*(D), M_q\otimes M_d}.
  \end{equation}
  is identifiable with the locally trivial fibration
  \begin{equation}\label{eq:u.homog.sp.fib}
    \begin{tikzpicture}[>=stealth,auto,baseline=(current  bounding  box.center)]
      \path[anchor=base] 
      (0,0) node (l) {$C_{U(qd)}(D)/C_{U(qd)}(M_q)$}
      +(4,.5) node (u) {$U(qd)/C_{U(qd)}(M_q)$}
      +(6,-1) node (r) {$U(qd)/C_{U(qd)}(D)$}
      ;
      \draw[right hook->] (l) to[bend left=6] node[pos=.5,auto] {$\scriptstyle $} (u);
      \draw[->>] (u) to[bend left=6] node[pos=.5,auto] {$\scriptstyle \mathrm{res}$} (r);
    \end{tikzpicture}
  \end{equation}
  (i.e. with the right-hand map in the above diagram).
\end{remark}

\begin{definition}\label{def:unif.thin}
  For an \emph{entourage} $D\subseteq X\times X$, member of a \emph{uniformity} \cite[Definition 7.1]{james_unif_1999} $(X,\cU)$ on $X$, a subset $Y\subseteq X$ is \emph{$D$-thin} if
  \begin{equation*}
    \exists\left(x\in X\right)
    \left(Y\subseteq D_x:=\left\{y\in X\ :\ (x,y)\in D\right\}\right).
  \end{equation*}
  A family $\cY=\left(Y_i\right)_i$ of subsets is $D$-thin if each member thereof is.

  The terminology applies in particular to compact Hausdorff spaces, with their unique \cite[Proposition 8.20]{james_unif_1999} topology-inducing uniformities. 
\end{definition}

\begin{proposition}\label{pr:embs.are.homotopies}
  Let $\cS\xrightarrowdbl{}\bS^3$ be the $M_q$ bundle associated via \Cref{rec:s3t2}\Cref{item:rec:s3t2:cz} to a reduced rational $\theta_{12}=\frac{p}{q}\in \bQ$ and $\bT:=\bD^2\times \bS^1$ one of the two solid tori in \Cref{eq:heeg.dec}.

  Suppose $q|n=qd$. Identifying $M_q$ with the left-hand tensorand $M_q\otimes M_d\cong M_{n}$, such embeddings exist only if the loop
  \begin{equation}\label{eq:zp.shift.diag}
    \bS^1
    \ni z
    \xmapsto{\quad\gamma\quad}
    \Ad_{\diag_{q-1}\left(z^p,1\cdots 1\right)\otimes I_d}
    \in
    \Emb^=_{M_q, M_n}
  \end{equation}
  is nullhomotopic through homotopies $\bS^1\times [0,1]\xrightarrow{h} \Emb^=_{M_q, M_n}$ with the family $\left(\mathrm{res}\; h(\bS^1\times \{s\})\right)_s \subseteq \Emb^=_{C^*(D), M_n}$ arbitrarily thin in the sense of \Cref{def:unif.thin}. 
\end{proposition}
\begin{proof}
  We work with the solid torus $\bD^2\times \bS^1\cong \bS^3_{\left[0,\frac 12\right]}$ containing $\bS^1_0$ as its core circle $\{0\}\times \bS^1$, with $w:=U_1^q$ as a coordinate along it. The restriction $\cS|_{\bD^2\cong \bD^2\times \{1\}}$ to an individual 2-disk slice embeds into the trivial $M_q$ bundle thereon, with 
  \begin{itemize}[wide]
  \item $M_q$ fiber over the punctured disk $\bD^2_{\times}:=\bD^2\setminus\{0\}$, generated by the $D$ of \Cref{eq:diag.subdiag} and $\diag_{q-1}(z,1\cdots 1)$ for the $z=U_2^q$ coordinate along $\partial \bD^2$;
  \item and $\bC^q$ exceptional fiber at $0\in \bD^2$, generated by the diagonal matrix $D$ alone.
  \end{itemize}
  The bundle $\cS_{\left[0,\frac 12\right]}$ over the entirety of $\bD^2\times \bS^1$ is obtained (cf. \cite[p.299]{rief_canc}) by
  \begin{itemize}[wide]
  \item pulling back the bundle over $\bD^2$ just described to one on the solid cylinder $\bD^2\times [0,1]$ along the first projection $\bD^2\times I\xrightarrowdbl{} \bD^2$;
  \item and identifying the two endpoints of $I$ so as to produce $\bD^2\times \bS^1$, with the gluing effected by the non-trivial automorphism $\Ad_{\diag_{q-1}\left(z^p,1\cdots 1\right)}$ on the restriction to $\bD^2\times \{0\}$. 
  \end{itemize}
    A locally trivial matrix bundle housing $\cS_{\left[0,\frac 12\right]}$, if extant, will be trivial by the remarks preceding \Cref{not:off.diag}; the fiber would thus be $M_{qd}$, $d\in \bZ_{>0}$ for $M_q$ embeds into $M_n$ precisely when $q|n$. In light of the description of $\cS_{\left[0,\frac 12\right]}$ just given, such an embedding exists if and only if the left-tensorand embedding
  \begin{equation}\label{eq:sres.to.mqd} 
    \cS|_{\bD^2}
    \lhook\joinrel\xrightarrow{\quad\iota\quad}
    \bD^2\times M_{qd}
    \cong
    \bD_{\times}^2\times \left(M_{q}\otimes M_d\right)
  \end{equation}
  is homotopic in the space of bundle embeddings \Cref{eq:sres.to.mqd} to $\iota\circ\Ad_{\diag_{q-1}\left(z^p,1\cdots 1\right)}$ with $z\in \bS^1$ being the coordinate around the boundary $\bS^1=\partial\bD^2$.

  Consider such a homotopy $H(-,-)$ with the second coordinate ranging over $s\in [0,1]$, and write $[X,Y]$ for the space of continuous maps $X\to Y$. This induces homotopies $H_t(-,s)$ in $\left[\bS^1,\Emb^=_{C^*(D),M_n}\right]$ at $t=0$ and $\left[\bS^1,\Emb^=_{M_q,M_n}\right]$ at $t\in (0,1]$ on the individual slices
  \begin{equation*}
    \left\{tz\ :\ z\in \bS^1\right\}\subset \bD^2
    ,\quad
    t\in [0,1].
  \end{equation*}
  $H_0\left(\bS^1,s\right)$ being constant for every $s$, the conclusion follows: for $t$ sufficiently close to $0$ the subsets
  \begin{equation*}
    \mathrm{res}\; H_t(\bS^1,s)
    \subseteq
    \Emb^=_{C^*(D),M_n}
  \end{equation*}
  will be uniformly small in $s\in [0,1]$. 
\end{proof}

We record a consequence, weaker than \Cref{th:s.not.ext} and to be superseded by the latter later. 

\begin{corollary}\label{cor:q2.div.n}
Let $\theta_{12}=\frac{p}{q}\in \bQ$ be a lowest-terms rational.
  
  The subhomogeneous $C^*$ bundle $\cS\xrightarrowdbl{} \bS^3$ of \Cref{rec:s3t2}\Cref{item:rec:s3t2:cz} associated to the quantum sphere $\bS^3_{\theta}$ embeds into a locally trivial $M_n$ bundle over either of the two solid tori $\bD^2\times \bS^1$ in the decomposition \Cref{eq:heeg.dec} only if $q^2|n$.
\end{corollary}
\begin{proof}
  Having conflated $M_q$ with the left-hand tensorand in $M_q\otimes M_d\cong M_{qd}$, we have an identification
  \begin{equation*}
    U(qd)/U(d)
    \ni
    \psi
    \xmapsto{\quad\cong\quad}
    \Ad_{\psi}|_{M_q}
    \in
    \Emb^=_{M_q,M_q\otimes M_d}
  \end{equation*}
  for the realization of the unitary group $U(d)$ as the unitary \emph{commutant} (or \emph{centralizer})
  \begin{equation}\label{eq:untr.cntrlz}
    \begin{aligned}
      C_{U(qd)}(M_q)
      &=
        \left\{u\in U(qd)\ :\ \Ad_u T=T,\ \forall T\in M_q\right\}\\
      &=
        1\otimes U(d)
        \ \le\ 
        U(q)\times U(d)
        \ \subset\ 
        M_q\otimes M_d
    \end{aligned}      
  \end{equation}
  of $M_q$ in $M_{qd}$. By \Cref{pr:embs.are.homotopies}, the sought-after embedding exists only when the loop (represented by) $\Ad_{\diag_{q-1}\left(z^p,1\cdots 1\right)}$ is trivial in $\pi_1 U(qd)/U(d)$. The long exact \emph{homotopy sequence} \cite[\S 17.3]{steen_fib_1999} attached to
  \begin{equation*}
    \{1\}
    \to
    SU(n)
    \lhook\joinrel\xrightarrow{\quad}
    U(n)
    \xrightarrowdbl{\quad\det\quad}
    \bS^1
    \to
    \{1\}
    ,\quad
    n\in \bZ_{>0}
  \end{equation*}
  identifies $\pi_1 U(n)$ with $\bZ$ with $\det$ inducing a $\pi_1$ isomorphism, which observation further applied in the long exact sequence associated to the principal $U(d)$-bundle $U(qd)\xrightarrowdbl{} U(qd)/U(d)$ gives $\pi_1 U(qd)/U(d)\cong \bZ/q$ so as to identify
  \begin{equation}\label{eq:ad.zp.1}
    \Ad_{\diag_{q-1}\left(z^p,1\cdots 1\right)}
    =
    \Ad_{\diag_{q-1}\left(z^p,1\cdots 1\right)\otimes I_d}
  \end{equation}
  with $pd$. By the assumed coprimality $\gcd(p,q)=1$, homotopic triviality is equivalent to $q|d=\frac{n}{q}$.
\end{proof}

\Cref{pr:embs.are.homotopies} can now be leveraged into a sharper equivalence criterion.

\begin{theorem}\label{th:map.through.ct.on.d}
  Let $\cS\xrightarrowdbl{}\bS^3$ be the $M_q$ bundle associated via \Cref{rec:s3t2}\Cref{item:rec:s3t2:cz} to a reduced rational $\theta_{12}=\frac{p}{q}\in \bQ$ and $\bT:=\bD^2\times \bS^1$ one of the two solid tori in \Cref{eq:heeg.dec}.

  Suppose $q|n=qd$. Identifying $M_q$ with the left-hand tensorand $M_q\otimes M_d\cong M_{n}$, such embeddings exist only if \Cref{eq:zp.shift.diag} is nullhomotopic through loops constant on $D$.

  Equivalently, the condition is that \Cref{eq:zp.shift.diag} be nullhomotopic in the fiber
  \begin{equation*}
    C_{U(qd)}(D)/C_{U(qd)}(M_q)
    \cong
    U(d)^q/U(d)
  \end{equation*}
  of the fibration \Cref{eq:u.homog.sp.fib} that contains it. 
\end{theorem}
\begin{proof}
  That the two formulations are mutually equivalent is tautological. 
  \begin{enumerate}[label={},wide]
  \item\textbf{($\Leftarrow$)} The \emph{homotopy lifting property} characteristic \cite[\S 5.5.5]{td_alg-top} of fibrations ensures the existence of
    \begin{equation*}
      \bS^1\times [0,1]
      \xrightarrow{\quad h\quad}
      U(qd)/C_{U(qd)}(M_q)
      ,\quad
      \begin{gathered}
        h|_{\bS^1\times\{0\}}=\Ad_{\diag_{q-1}\left(z^p,1\cdots 1\right)\otimes I_d}\\
        h|_{\bS^1\times\{1\}}=1\\
        \forall t\left(\mathrm{res}\; h|_{\bS^1\times\{t\}}\text{ is constant}\right).
      \end{gathered}
    \end{equation*}
    To conclude, collect the slices $\Ad_{h|_{\bS^1\times\{t\}}}$, $t\in [0,1]$ into a single embedding \Cref{eq:sres.to.mqd} with $\Ad_{h|_{\bS^1\times\{t\}}}$ operating respectively along the circle $(1-t)\bS^1$.

  \item\textbf{($\Rightarrow$)} Equip the (compact) total space of \Cref{eq:u.homog.sp.fib} with a Riemannian structure, providing the corresponding \emph{geodesic distance} \cite[Definition 7.2.4]{dcrm_riem_1992}. We know from \Cref{pr:embs.are.homotopies} that the loops $\gamma_s$, $s\in [0,1]$ through which \Cref{eq:zp.shift.diag} nullhomotopes can be selected so as to have
    \begin{equation*}
      \forall\left(s\in [0,1]\right)
      \left(\mathrm{diam}\gamma_{s}\left(\bS^1\right)(D)<\varepsilon\right)
    \end{equation*}
    for arbitrarily small $\varepsilon$. Were it small enough (and given the compactness of fibers and base alike in \Cref{eq:u.homog.sp.fib}), \cite[Theorem 2.7.12]{bg_diff-geom_1e_en_1988} ensures the existence of unique, smooth-varying curves
    \begin{equation*}
      [0,1]
      \xrightarrow{\quad\alpha_{z,s}\quad}
      U(qd)/C_{U(qd)}(M_q)
      ,\quad
      \begin{aligned}
        \alpha_{z,s}(0)
        &=
          \gamma_s(z)\\
        \alpha_{z,s}(1)
        &\in
          \mathrm{res}^{-1}\left(\gamma_s(1)|_{C^*(D)}\right)
      \end{aligned}
    \end{equation*}
    their respective origins to their corresponding unique closest points in the fibers of \Cref{eq:u.homog.sp.fib} just displayed. Flowing along the $\alpha_{z,s}$ will now implement a homotopy from the original $\left(\gamma_s\right)_s$ into another, $\left(\wt{\gamma}_s\right)_s$ say, with the individual loops $\wt{\gamma}_s(\bS^1)$ entirely contained in individual fibers. Trivializing the pullback of the bundle \Cref{eq:u.homog.sp.fib} through the path
    \begin{equation*}
      \mathrm{res}\left(\wt{\gamma}_s(\bS^1)\right)_s
      \in
      \left[[0,1],U(qd)/C_{U(qd)}(D)\right]      
    \end{equation*}
    in the base, the existence of such a nullhomotopy is indeed equivalent to the triviality of the original loop in the fiber. 
  \end{enumerate}
\end{proof}

\pf{th:s.not.ext}
\begin{th:s.not.ext}
  Embeddability into locally trivial subhomogeneous bundles on the one hand and matrix bundles on the other are equivalent \cite[Proposition 2.3]{2605.10752v1}, so we assume an $M_{n=qd}$ fiber throughout. 
  
  Naturally, one implication needs no elaboration; for the other, \Cref{th:map.through.ct.on.d} reduces the problem to showing that as soon as $q>1$ (i.e. $\theta_{12}\not\in \bZ$) \Cref{eq:zp.shift.diag} cannot be nullhomotopic in the fiber
  \begin{equation*}
    C_{U(qd)}(D)/C_{U(qd)}(M_q)
    \cong
    U(d)^q/U(d)
  \end{equation*}
  (we may as well assume $n=qd$ is divisible by $q$). The portion
  \begin{equation*}
    \cdots
    \xrightarrow{\quad}
    \pi_1\; U(d)
    \xrightarrow{\quad}
    \pi_1\; U(d)^q
    \xrightarrow{\quad}
    \pi_1\; U(d)
    \xrightarrow{\quad}
    \pi_0\cdots
  \end{equation*}
  of the homotopy sequence attached to $C_{U(qd)}(D)\cong U(d)^q$ regarded as a principal $\left(C_{U(qd)}(M_q)\cong U(d)\right)$-bundle (together with the fact that $\det$ induces a $\pi_1$-isomorphism $U(\bullet)\to \bS^1$) shows that the aforementioned nullhomotopy amounts precisely to \Cref{eq:zp.shift.diag} operating with distinct determinants on the $q>1$ eigenspaces of $D\otimes I_d\in M_q\otimes M_d$ after translation to the base fiber $U(d)^q/U(d)$ above the basepoint
  \begin{equation*}
    U(d)^q
    \in
    U(dq)/U(d)^q
    \cong
    U(dq)/C_{U(dq)}(D).
  \end{equation*}
  Indeed:
  \begin{equation*}
    \left(\diag_{1}(1\cdots 1)\otimes I_d\right)
    \cdot
    \left(\diag_{q-1}\left(z^p,1\cdots 1\right)\otimes I_d\right)
    =
    \diag_{0}\left(z^p,1\cdots 1\right)\otimes I_d
    \in
    C_{U(qd)}(D),
  \end{equation*}
  with respective determinants $z^{pd}$ on one eigenspace and 1 on every other. 
\end{th:s.not.ext}

As the torus bundles of concern in \Cref{th:s.not.ext} are obtained by restricting the sphere-based ones of \Cref{rec:s3t2}, an immediate consequence of the theorem is its analogue for quantum 3-spheres.

\begin{corollary}\label{cor:s3.theta}
  Let $\theta_{12}=\frac{p}{q}\in \bQ$ be a lowest-terms rational.
  
  The subhomogeneous $C^*$ bundle $\cS\xrightarrowdbl{} \bS^3$ of \Cref{rec:s3t2}\Cref{item:rec:s3t2:cz} associated to the quantum sphere $\bS^3_{\theta}$ embeds into a locally trivial subhomogeneous $C^*$ bundle if and only if $q=1$.  \qedhere
\end{corollary}

\pf{th:sn.theta}
\begin{th:sn.theta}
  That the inclusion $Z(C^n_{\theta})=:Z^n_{\theta}\le C^n_{\theta}$ is dual to a continuous subhomogeneous $C^*$ field $\cS\xrightarrowdbl{}X$ follows from \Cref{eq:a.is.end.e} (the analogous claim for tori) together with \Cref{eq:cntheta}. The finite-type claim follows (as argued in \cite[proof of Theorem 3.3]{MR5049282}, say) from \cite[Proposition 2.1]{zbMATH03635101} and the fact that the strata are metrizable, finite unions of path-connected spaces, and have a uniform bound on their \emph{covering dimensions} \cite[Definition 8-1]{nk_dim}.

  The generic fiber $\cS_x$ being a matrix algebra (because this is so for quantum tori), the tracial expectation will indeed be unique if it exists. That it does exist follows from the realization \Cref{eq:cntheta}, which makes it clear that the usual normalized traces on the individual generic fibers will glue compatibly across singular loci. The fact that \Cref{eq:K.ct} is optimal is part of \cite[Theorem 4.1]{bg_cx-exp}, and in the present case that value is $q$ for the common dimension $q\times q$ of the generic matrix-algebra fibers.

  It remains to address the statement's item \Cref{item:th:sn.theta:not.emb}, which reduces to the case $n=2$: for any individual $\theta_{ij}$ the quotient
  \begin{equation*}
    C^n_{\theta}
    \xrightarrowdbl{\quad}
    C\left(\bS^3_{\theta_{ij}}\right)
  \end{equation*}
  obtained by annihilating all but the two generators $T_{i,j}$ (where a slight notational abuse substitutes the single number $\theta_{ij}$ for the corresponding skew-symmetric $2\times 2$ matrix) is dual to a restriction of $\cS\xrightarrowdbl{}X$ to an $\bS^3\subseteq X$. As to quantum 3-spheres, the embeddability issue is addressed by \Cref{cor:s3.theta}. 
\end{th:sn.theta}

%%%%%%%%%%%%%%%%%%%%%%%%%%%%%%%%
%%%%%%%%%%%%%%%%%%%%%%%%%%%%%%%%

\addcontentsline{toc}{section}{References}
%\bibliography{bib}{}
%\bibliographystyle{plain}

% BEGIN INSERTED BBL (irremovable-bundle-singularities-xv1.bbl)
\def\polhk#1{\setbox0=\hbox{#1}{\ooalign{\hidewidth
  \lower1.5ex\hbox{`}\hidewidth\crcr\unhbox0}}}

% END INSERTED BBL

\Addresses

\end{document}